\documentclass[12pt,oneside,reqno]{amsart}
\usepackage[utf8]{inputenc}
\usepackage[displaymath, mathlines, pagewise]{lineno}

\usepackage{amssymb,latexsym}
\usepackage{amsmath,amsfonts,amsthm,amscd,amsxtra,enumerate}
\usepackage{mathtools}
\usepackage[all]{xy}
\usepackage{amsmath,latexsym,amssymb, enumerate, amsthm, epsfig, hyperref} 
\usepackage{caption}

\usepackage[margin=1in]{geometry}
\usepackage{epsfig}
\usepackage{graphicx}
\usepackage{multicol}

\usepackage{fancyhdr}

\fancypagestyle{mystyle}{
\fancyhf{}

\fancyfoot[R]{\thepage}
}

\fancypagestyle{plain}{
\fancyhf{}

\fancyfoot[R]{\thepage}
}

\newtheorem{theorem}{Theorem}[section]
\newtheorem{lemma}[theorem]{Lemma}

\newtheorem{corollary}[theorem]{Corollary}

\newtheorem{remark}[theorem]{Remark}
\newtheorem{proposition}[theorem]{Proposition}

\DeclareMathOperator{\im}{im}

\DeclareMathOperator{\Hom}{Hom}

\DeclareMathOperator{\depth}{depth}

\usepackage{fullpage}
\DeclareFontFamily{U}{mathx}{\hyphenchar\font45}
\DeclareFontShape{U}{mathx}{m}{n}{ <5> <6> <7> <8> <9> <10> <10.95>
  <12> <14.4> <17.28> <20.74> <24.88> mathx10 }{}
\DeclareSymbolFont{mathx}{U}{mathx}{m}{n}
\DeclareFontSubstitution{U}{mathx}{m}{n}
\DeclareMathAccent{\widecheck}{0}{mathx}{"71}
\DeclareMathAccent{\wideparen}{0}{mathx}{"75}

\theoremstyle{plain}
\begin{document}

\title{\textbf{Generic doublings of almost complete intersections of codimension 3 }}

\author[Jai Laxmi]{Jai Laxmi}
\address{Department of Mathematics, University of Connecticut, Storrs, CT 06269 }
\email{jai.laxmi@uconn.edu, laxmiuohyd@gmail.com}

\subjclass[2010]{13H10, 15A66,15A75.}
\keywords{ Cohen Macaulay rings, Gorenstein ring, spinor coordinates, spinor structures, almost complete intersection, reflexive modules, conormal modules,normal modules.}

%
%


\maketitle

\pagenumbering{arabic}
\begin{abstract}
We study Gorenstein ideals of codimension $4$ derived from generic doublings of almost complete intersection perfect ideals of codimension $3$. 
 We also investigate spinor coordinates of such Gorenstein ideals with $8$ and $9$ generators. For an ideal $J$ of commutative ring $R$, the $R/J$ module $J/J^2$ is called conormal module and $R/J$-dual of $J/J^2$ is called normal module. We  study properties of conormal and normal modules of almost complete intersection perfect ideals of codimension $3$.  
\end{abstract}

\section{Introduction}
The problem of classifying  Gorenstein ideals of codimension $4$ was around ever since Buchsbaum and Eisenbud classified  Gorenstein ideals of codimension $3$ in \cite{BE1}. The first results related to the structure of codimension $4$ case were obtained by Kustin and Miller \cite{KM}. We study Gorenstein ideals of codimension $4$ obtained from generic doubling of almost complete intersection perfect ideals of codimension $3$. 

Christensen, Veliche and Weyman in \cite{CVW} gave structures of two types of generic form of almost complete intersection ideals of codimension $3$, of Cohen Macaulay even  and odd type, see \cite[Prop. 2.2,2.4]{CVW}. Moreover in \cite[Theorem 4.1]{CVW}, they show that any minimal free resolution of almost complete intersection perfect ideal $J_n$ over a commutative local ring $R$, of codimension $3$ and Cohen Macaulay type $n$, is a specialization of one of the type, either Cohen Macaulay even or odd type stated in \cite[Prop. 2.2,2.4]{CVW}. So it is enough to study generic doublings of two types of generic form of almost complete intersection ideals of codimension $3$, of Cohen Macaulay even  and odd type. We denote generic form of almost complete intersection ideals of Cohen Macaulay type $n$ given in \cite[Prop. 2.2,2.4]{CVW} as $J_n$ over a polynomial ring $R$ in generic variables.

In Section \ref{almostci} we recall structure of generic form of almost complete intersection ideals of codimension $3$, of Cohen Macaulay even  and odd type, see \cite[Prop. 2.2,2.4]{CVW}.

We set generic almost complete intersection ring of type $n$ as $S_n=R/J_n$ with a canonical module $\omega_{S_n}$. Section \ref{Hom} deals with equivariant generators of $\Hom_{S_n}(\omega_{S_n},S_n)$.  The $S_n$-module $J_n/J_n^2$ is called conormal module and $S_n$-dual $J_n/J_n^2$ is called normal module. In Section \ref{NC} we study properties of conormal and normal modules of $J_n$. Matsouka \cite{M} proved that there is an embedding $\varphi_n:\omega_{S_n}\rightarrow S_n^4$ such that $S^4_n/\varphi_n(\omega_{S_n})\simeq J_n/J_n^2$. Using embedding $\varphi_n$ we show that $S_n$-module  $\Hom_{S_n}(\omega_{S_n},S_n)$ has $4$ generators and conormal module of $J_n$ is a reflexive $S_n$-module. Moreover we observe  that $\Hom_{S_n}(\omega_{S_n},S_n)$ is a maximal Cohen Macaulay $S_n$ module.

In Section \ref{GenD} we construct generic doublings of $S_n$ of Cohen Macaulay odd and even type using generators of $\Hom_{S_n}(\omega_{S_n},S_n)$. Such generic doublings give Gorenstein rings of codimension $4$.   
By \cite[Theorem 4.2]{CLW}, there exist spinor structures on generic doublings of $S_n$. We investigate spinor coordinates of Gorenstein ideals with $9$ generators obtained from generic doublings of $S_5$. There are  two examples in  Gorenstein ideals with $9$ generators in \cite[Example 5.3]{CLW} where none of the spinor coordinates are among minimal generators of ideals. In Proposition \ref{spin9} we show that there are $5$ spinor coordinates among minimal generators of Gorenstein ideals obtained from doublings of $S_5$. In Section \ref{spin8} we study spinor coordinates of generic doublings of $S_4$.

\section{Structure of almost complete intersection rings}\label{almostci}
We recall structures of almost complete intersections of codimension $3$  given by Christensen, Veliche, and Weyman in \cite{CVW}. By \cite[Theorem 4.1]{CVW}, any minimal free resolution of almost complete intersection perfect ideal $J_n$ over a commutative local ring $R$, of codimension $3$ and Cohen Macaulay type $n$, is a specialization of one of the type, either Cohen Macaulay even or odd type stated in \cite[Prop. 2.2,2.4]{CVW}. 


\subsubsection{Structure of almost complete intersection for $n=2m+1$} 
We record results from \cite[Prop. 2.2]{CVW}. Let $\mathbb{K}$ be a field. 
Consider a $(2m+1)\times (2m+1)$ generic skew symmetric matrix $C=(c_{ij})$, and $3\times (2m+1)$ generic matrix $U=(u_{kl})$ with $1\leq k\leq 3$ and $1\leq l\leq 2m+1$. Let $R$ be a polynomial ring over $\mathbb{K}$ on the entries of $C$ and $U$. Set $F=R^{2m+1}$ and $G=R^3$.  

Set $J_{2m+1}=\langle x_1,x_2,x_3,x_4\rangle$  where $x_1=C^{m-1}\wedge u_1\wedge u_2\wedge u_3$, $x_2=C^m\wedge u_1$, $x_3=C^m\wedge u_2$ and $x_4=C^m\wedge u_3$ where $u_i$ are ith row of $U$. Then 

\begin{equation}\label{almostodd}
0\xrightarrow{} F\xrightarrow{d_3} F^*\oplus G^*\xrightarrow{d_2} R\oplus G\xrightarrow{d_1} R\rightarrow R/J_{2m+1}\rightarrow 0 
\end{equation}

is minimal free resolution of $R/J_{2m+1}$ where 
$$d_1=\begin{bmatrix}
x_1 & x_2 & x_3 & x_4\\
\end{bmatrix},\;\;\;\; d_3=\begin{bmatrix}
C\\
U
\end{bmatrix},$$ $$d_2=\begin{bmatrix}
w_1 & w_2 & \ldots & w_{2m+1}& 0 & 0 & 0 \\
v_{1,1} & v_{1,2} & \ldots & v_{1,2m+1} & x_3 & x_4 & 0\\
v_{2,1} & v_{2,2} & \ldots & v_{2,2m+1}  & -x_2 & 0 & x_4 \\
v_{3,1} & v_{3,2} & \ldots & v_{3,2m+1}  & 0 & -x_2&  x_3\\
\end{bmatrix}$$
Here $w_i=\pm {\rm Pf}(\hat{i})$ and $v_{i,\gamma}=\sum\limits_{j,k}\pm \Delta^{j,k}_{\alpha,\beta}{\rm Pf}(\hat{i},\hat{j},\hat{k})$ where ${\rm Pf}(\hat{i},\hat{j},\hat{k})$ is Pfaffians of $C$ with omitted $i,j,k$ rows and columns, and $\Delta^{j,k}_{\alpha,\beta}$ is $2\times 2$ minor of $U$ involving $\alpha,\beta$ rows and $j,k$ columns. Also $\gamma$ is complement of $\alpha,\beta$ in $\{1,2,3\}$. The Schubert varieties are normal domain, so  almost complete intersection $R/J_{2m+1}$ are also.

\subsubsection{Structure of almost complete intersection for $n=2m$}
We record results from \cite[Prop. 2.4]{CVW}. Let $\mathbb{K}$ be a field. Consider a $2m\times 2m$ generic skew symmetric matrix $C=(c_{ij})$, and $3\times 2m$ generic matrix $U=(u_{kl})$ with $1\leq k\leq 3$ and $1\leq l\leq 2m$. Let $R$ be a polynomial ring over $\mathbb{K}$ on the entries of $C$ and $U$. Set $F=R^{2m}$ and $G=R^3$.

 
 Set $J_{2m}=\langle x_1,x_2,x_3,x_4\rangle$ where $ x_1=C^m$, $x_2=C^{m-1}\wedge u_1\wedge u_2$, $x_3=C^{m-1}\wedge u_1\wedge u_3$ and $x_4=C^{m-1}\wedge u_2\wedge u_3$. Then 

\begin{equation}\label{almosteven}
0\xrightarrow{} F\xrightarrow{d_3} F^*\oplus G^*\xrightarrow{d_2} G^*\oplus R\xrightarrow{d_1} R\rightarrow R/J_{2m}\rightarrow 0 
\end{equation}

is minimal free resolution of $R/J_{2m}$ where
$$d_1=\begin{bmatrix}
x_1 & x_2 & x_3 & x_4\\
\end{bmatrix},\;\;\;\; d_3=\begin{bmatrix}
C\\
U
\end{bmatrix},$$

$$d_2=\begin{bmatrix}
w_1 & w_2 & \ldots & w_{2m}& x_1 & x_2 & x_3 \\
v_{\{2,3\},1} & v_{\{2,3\},2} & \ldots & v_{\{2,3\},2m} & 0 & 0 & -x_1\\
-v_{\{1,3\},1} & -v_{\{1,3\},2} & \ldots & -v_{\{1,3\},2m}  & 0& -x_1 & 0 \\
v_{\{1,2\},1} & v_{\{1,2\},2} & \ldots & v_{\{1,2\},2m}  & -x_1 & 0&  0\\
\end{bmatrix}.$$
Here $w_i=\sum\limits_{j,k,l}\pm \Delta^{j,k,l}{\rm Pf}(\hat{i},\hat{j},\hat{k},\hat{l})$ and $v_{\{\alpha,\beta\},i}=\sum\limits_{j}u_{\gamma,j}{\rm Pf}(\hat{i},\hat{j})$ where ${\rm Pf}(\hat{i},\hat{j},\hat{k},\hat{l})$ is Pfaffian of $C$ with $i,j,k,l$ rows and columns omitted, and $\Delta^{i,j,k}$ is a $3\times 3$ minor of $U$ involving $i,j,k$ columns. Also $\gamma$ is complement of $\{\alpha, \beta\}$ in $\{1,2,3\}$. The Schubert varieties are normal domain, so  almost complete intersection $R/J_{2m}$ are also. 


\section{Generators of $\Hom_{S_n}(\omega_{S_n},S_n)$}\label{Hom}
In this section we discuss generators of $\Hom_{S_n}(\omega_{S_n},S_n)$. Assume notation stated in  Section \ref{almostci}. Let $S_n=R/J_n$ with a canonical module $\omega_{S_n}$. Since $S_n$ is normal domain, $(S_n)_p$ is regular for all prime ideals $p$ of height one. Thus $(S_n)_p$ is complete intersection for height one prime ideals $p$. Then by \cite[Theorem 1]{K} $\omega_{S_n}$ is reflexive $S_n$ module and $J_n/J_n^2$ is a torsion free $S_n$ module.

Let us study generators of $\omega_{S_n}$. For a matrix $A$, $A^t$ denotes the transpose of $A$. We have $ v_{1,j}x_2+v_{2,j}x_3+v_{3,j}x_4-w_jx_1=0$ and $w_jx_1-v_{\{2,3\},j}x_2 +v_{\{1,3\},j}x_3-v_{\{1,2\},j}x_4 =0$ for $n=2m+1$ and $n=2m$ respectively for $1\leq j\leq n$ since $d_1d_2=0$. Set $L_n=\langle x_2,x_3,x_4 \rangle$. Then $w_j\in (L_n:J_n)$ for all $j$. Since $\omega_{S_n}$ is the first Koszul homology module on generators $x_1,x_2,x_3$ and $x_4$ of $J_n$, see \cite{HK}, we get $\omega_{S_n}\simeq (L_n:J_n)/L_n$.


Set matrices in $S_n$ as

$$\varphi_{2m+1}=\begin{bmatrix}
\bar{w}_1 & \bar{w}_2 & \ldots & \bar{w}_{2m+1}\\
\bar{v}_{1,1} & \bar{v}_{1,2} & \ldots & \bar{v}_{1,2m+1}\\
\bar{v}_{2,1} & \bar{v}_{2,2} & \ldots & \bar{v}_{2,2m+1} \\
\bar{v}_{3,1} & \bar{v}_{3,2} & \ldots & \bar{v}_{3,2m+1} \\
\end{bmatrix},\;\;\;
\varphi_{2m}=\begin{bmatrix}
\bar{w}_1 & \bar{w}_2 & \ldots & \bar{w}_{2m} \\
\bar{v}_{\{2,3\},1} & \bar{v}_{\{2,3\},2} & \ldots & \bar{v}_{\{2,3\},2m}\\
-\bar{v}_{\{1,3\},1} & -\bar{v}_{\{1,3\},2} & \ldots & -\bar{v}_{\{1,3\},2m}\\
\bar{v}_{\{1,2\},1} & \bar{v}_{\{1,2\},2} & \ldots & \bar{v}_{\{1,2\},2m}\\
\end{bmatrix}
$$
where\;\; $\bar{}$\;\; denotes going mod $J_n$.

Let $\{e_1,e_2,e_3,e_4\}$ be a basis of $S_n^4$. Then by \cite[Prop. 1]{M} there is a short exact sequence 

\begin{equation}\label{exact2}
0\xrightarrow{}\omega_{S_n}\xrightarrow{\varphi_n} S_n^4\xrightarrow{\pi_n} J_n/J_n^2\rightarrow 0.
\end{equation}
 with $\pi_n(\sum\limits_{i=1}^4 t_i e_i)=\sum\limits_{i=1}^4 T_i x_i $ where $T_i$ is a representative of $t_i$ in $R$.

We discuss the equivariant form of generators of $\Hom_{S_n}(\omega_{S_n},S_n)$. These generators play crucial role in construction of generic doublings of almost complete intersection. 

\begin{proposition}\label{prop1}
 Consider resolutions (\ref{almostodd}) and (\ref{almosteven}). Set the transpose of $\varphi_n$ as $H_n$. Then $\im(H_n)\subseteq \Hom_{S_n}(\omega_{S_n},S_n)$.
   \end{proposition}
\begin{proof}
By a change of basis, a free presentation of $\omega_S$ is 
\begin{equation}\label{3.1}
(F\oplus G)\otimes S_n\xrightarrow{d_3^t\otimes S_n} F^*\otimes S_n\rightarrow \omega_{S_n}\rightarrow 0.
\end{equation}
The equivariant form of the generators of $\Hom_{S_n} (\omega_{S_n}, S_n)$ is the kernel of the map
$$F\otimes S_n\xrightarrow{d_3\otimes S_n} (F^*\oplus G^*)\otimes S_n.$$

For $n=2m$, consider the module $H_{2m}$ generated by the image of
$\bigwedge\limits^{2m}F\otimes G\mapsto F\otimes \bigwedge\limits^{2m-2}F\otimes (F\otimes G)$
where we identify the factors $\bigwedge\limits^{2m-2}F$ and $F\otimes G$ with sub representations in $R$,
and by the image of $\bigwedge\limits^{2m}F\otimes \bigwedge\limits^3 G\mapsto F\otimes \bigwedge\limits^{2m-4}F\otimes (\bigwedge^3 F\otimes \bigwedge\limits^3 G)$,
where we identify the factors  $\bigwedge\limits^{2m-4}F$ and $\bigwedge\limits^3 F\otimes \bigwedge\limits^3 G$ with sub representations in R. 

 Set basis of $F$ and $G$ as $\mathcal{B}_{F}=\{f_1,\ldots,f_{2m}\}$ and $\mathcal{B}_G=\{e_1,e_2,e_3\}$ respectively. Let $\sum_{r+s}$ be a symmetric group on $\{1,2,\ldots,r+s\}$. Define map $$\Delta:\bigwedge\limits^{r+s}F\rightarrow \bigwedge\limits^r F\otimes \bigwedge\limits^s F$$ 
 as $$\Delta(f_1\wedge\cdots \wedge f_{r+s}):=\sum\limits\limits_{\sigma \in \sum^{r,s}_{r+s}}(-1)^{{\rm sgn}\;\; \sigma}f_{\sigma(1)}\wedge\cdots \wedge f_{\sigma(r)}\otimes f_{\sigma(r+1)}\wedge \cdots \wedge f_{\sigma(r+s)}$$ where $\sum^{r,s}_{r+s}=\{\sigma\in \sum_{r+s}| \sigma(1)<\cdots\sigma(r); \sigma(r+1)<\cdots< \sigma(r+s) \}$. Then we have

\[\bigwedge \limits^{2m}F\otimes G\xrightarrow{\Delta\otimes 1}F\otimes \bigwedge\limits^{2m-1}F\otimes G\xrightarrow{1\otimes \Delta\otimes 1}F\otimes \bigwedge\limits^{2m-2}F\otimes(F\otimes G)\]
such that 
\[(\Delta\otimes 1)(f_1\wedge f_2\wedge\cdots\wedge f_{2m}\otimes e_i)=\sum\limits_{\sigma\in \sum_{2m}^{1,2m-1}}(-1)^{{\rm sgn}\; \sigma}f_{\sigma(1)}\otimes f_{\sigma(2)}\wedge \cdots \wedge f_{\sigma(2m)}\otimes e_i,\]
\begin{align*}
(1\otimes \Delta \otimes 1)((-1)^{{\rm sgn}\;\sigma}f_{\sigma(1)}\otimes f_{\sigma(2)}\wedge \cdots \wedge f_{\sigma(2m)}\otimes e_i)\\
=\sum\limits_{\tau\in \sum^{1,2m-2}_{2m-1},\tau(\sigma(1))=\sigma(1)}(-1)^{{\rm sgn}\; (\tau\sigma)}f_{\sigma(1)}\otimes f_{\tau\sigma(2)}\wedge \cdots \wedge f_{\tau\sigma(2m-1)}\otimes f_{\tau\sigma(2m)}\otimes e_i
\end{align*}

where $\tau\sigma(2)<\tau\sigma(3)<\ldots< \tau\sigma(2m-1)$, and

\[\bigwedge\limits^{2m}F\otimes \bigwedge\limits^3 G\xrightarrow{\Delta\otimes 1}F\otimes \bigwedge\limits^{2m-1}F\otimes \bigwedge\limits^3 G\xrightarrow{1\otimes \Delta\otimes 1}F\otimes \bigwedge\limits^{2m-4}F\otimes (\bigwedge\limits^3 F\otimes \bigwedge\limits^3 G)\]

\[(\Delta\otimes 1)(f_1\wedge f_2\wedge\cdots\wedge f_{2m}\otimes e_1\wedge e_2\wedge e_3)=\sum\limits_{\sigma\in \sum_{2m}^{1,2m-1}}(-1)^{{\rm sgn} \;\sigma}f_{\sigma(1)}\otimes f_{\sigma(2)}\wedge \cdots \wedge f_{\sigma(2m)}\otimes e_1\wedge e_2\wedge e_3.\]

{\tiny
\begin{align*}
(1\otimes \Delta \otimes 1)((-1)^{{\rm sgn}\;\sigma}f_{\sigma(1)}\otimes f_{\sigma(2)}\wedge \cdots \wedge f_{\sigma(2m)}\otimes e_1\wedge e_2\wedge e_3)\\
=\sum\limits_{\tau\in \sum^{1,2m-2}_{2m-1},\tau(\sigma(1))=\sigma(1)}(-1)^{{\rm sgn} (\tau\sigma)}f_{\sigma(1)}\otimes f_{\tau\sigma(2)}\wedge \cdots \wedge f_{\tau\sigma(2m-3)}\otimes f_{\tau\sigma(2m-2)}\wedge f_{\tau\sigma(2m-1)}\wedge f_{\tau\sigma(2m)}\otimes e_1\wedge e_2\wedge e_3
\end{align*}}

where $\tau\sigma(2)<\tau\sigma(3)<\ldots< \tau\sigma(2m)$. 

For indexing set $\mathcal{L}\subset\{1,\ldots,n\}$, we donote Pfaffian of $C$ involving $\mathcal{L}$ rows and columns of $C$ as ${\rm Pf}(\mathcal{L})$. 
For $\sigma\in \sum_{2m}^{1,2m-1}$ set 
$$v_{\{\alpha,\beta\},\sigma(1)}=\sum\limits_{\tau\in \sum^{1,2m-2}_{2m-1},\tau(\sigma(1))=\sigma(1)}(-1)^{{\rm sgn}(\tau\sigma)}u_{\gamma,\tau\sigma(2m)}{\rm Pf}(\widehat{\sigma(1)},\widehat{\tau\sigma(2m)}),$$

{\tiny
$$w_{\sigma(1)}=\sum_{\tau\in \sum^{1,2m-2}_{2m-1},\tau(\sigma(1))=\sigma(1)}(-1)^{{\rm sgn}(\tau\sigma)}{\rm Pf}(\widehat{\sigma(1)},\widehat{\tau\sigma(2m-2)},\widehat{\tau\sigma(2m-1)},\widehat{\tau\sigma(2m)})\Delta_{1,2,3}^{\tau\sigma(2m-2),\tau\sigma(2m-1),\tau\sigma(2m)}$$}
where $\gamma$ is the complement of $\{\alpha,\beta\}$ in the set $\{1,2,3\}$, and $ \Delta^{j,k,l}$ is a $3\times 3$ minors of $3\times(2m)$ matrix with rows $u_1$, $u_2$, $u_3$ of matrix $U$ involving columns $j,k,l$. Then the matrix presentation of $H_{2m}$ is the transpose of $\varphi_{2m}$

For $n=2m+1$ consider the module $H_{2m+1}$ generated by the image of
$\bigwedge\limits^{2m+1}F\mapsto F\otimes \bigwedge\limits^{2m}F$ where we identify the factor $\bigwedge\limits^{2m}F$ with subrepresentation in $R$,
and by the image of
$\bigwedge\limits^{2m+1}F\otimes\bigwedge\limits^2 G\mapsto F\otimes \bigwedge\limits^{2m-2}F\otimes \bigwedge\limits^2 F\otimes \bigwedge\limits^2 G$ where we identify the factors  $\bigwedge\limits^{2m-2}F$ and $\bigwedge\limits^2 F\otimes \bigwedge\limits^2 G$ with sub representations in $R$. Set basis of $F$ and $G$ as $\mathcal{B}_{F}=\{f_1,\ldots,f_{2m+1}\}$ and $\mathcal{B}_G=\{e_1,e_2,e_3\}$ respectively. We have the following maps:
 
\[\bigwedge^{2m+1}F\xrightarrow{\Delta\otimes 1} F\otimes \bigwedge^{2m}F\]

\[(\Delta\otimes 1)(f_1\wedge \cdots\wedge f_{2m+1})=\sum\limits_{\sigma\in \sum_{2m+1}^{1,2m}}(-1)^{{\rm sgn}\; \sigma}f_{\sigma(1)}\otimes f_{\sigma(2)}\wedge \cdots\wedge f_{\sigma(2m+1)}\]


\[\bigwedge^{2m+1}F\otimes\bigwedge^2 G\xrightarrow{\Delta\otimes 1} F\otimes \bigwedge^{2m}F\otimes \bigwedge^2 G  \xrightarrow{1\otimes \Delta\otimes 1}F\otimes \bigwedge^{2m-2}F\otimes (\bigwedge^2 F\otimes \bigwedge^2 G)\]

\[(\Delta\otimes 1)(f_1\wedge f_2\wedge\cdots\wedge f_{2m+1}\otimes e_i\wedge e_j)=\sum\limits_{\sigma\in \sum^{1,2m}_{2m+1}}(-1)^{{\rm sgn}\; \sigma}f_{\sigma(1)}\otimes f_{\sigma(2)}\wedge \cdots \wedge f_{\sigma(2m+1)}\otimes e_i\wedge e_j\]

\begin{align*}
(1\otimes \Delta \otimes 1)((-1)^{sgn(\sigma)}f_{\sigma(1)}\otimes f_{\sigma(2)}\wedge \cdots \wedge f_{\sigma(2m+1)}\otimes e_i\wedge e_j)\\
=\sum\limits_{\tau\in \sum_{2m}^{1,2m-1},\tau(\sigma(1))=\sigma(1)}(-1)^{{\rm sgn} (\tau\sigma)}f_{\sigma(1)}\otimes f_{\tau\sigma(2)}\wedge \cdots  \wedge f_{\tau\sigma(2m-1)}\otimes f_{\tau\sigma(2m)}\wedge f_{\tau\sigma(2m+1)}\otimes e_i\wedge e_j
\end{align*}
where $\tau\sigma(2)<\tau\sigma(3)<\ldots< \tau\sigma(2m-1)$ and $\tau\sigma(2m)<\tau\sigma(2m+1)$. \\

Set $$v_{\sigma(1),\gamma}=\sum_{\tau\in \sum_{2m}^{1,2m-1},\tau\sigma(1)=\sigma(1)}(-1)^{{\rm sgn}(\tau\sigma)}\Delta_{\alpha,\beta}^{\tau\sigma(2m),\tau\sigma(2m+1)}{\rm Pf}(\widehat{(\sigma(1)},\widehat{\tau\sigma(2m)},\widehat{\tau\sigma(2m+1)})$$
where $\gamma$ is complement of $\{\alpha,\beta\}$ in the set $\{1,2,3\}$, where $\tau\sigma(2)<\tau\sigma(3)<\ldots< \tau\sigma(2m-1)$ and $\tau\sigma(2m)<\tau\sigma(2m+1)$. Set $w_{\sigma(1)}={\rm Pf}(\widehat{\sigma(1)})$. Then the matrix presentation of $H_{2m+1}$ is the transpose $\varphi_{2m+1}$. 

Since $d_2d_3=0$, this forces $\im(H_n)\subseteq \ker(d_3\otimes S_n)$ which is same as $\im(H_n)\subseteq \Hom_{S_n}(\omega_{S_n},S_n)$. 
\end{proof}

\section{Normal and conormal module of almost complete intersections}\label{NC}

For an $S_n$-module $M$, $M^*=\Hom_{S_n}(M,S_n)$, and $Ass(M)$ denotes the set of associated primes of $M$. We denote ${\rm ht}(p)$ as the height of prime ideal $p$. For ideal $J_n$, the $S_n$-module $J_n/J_n^2$ is called conormal module of $J_n$, and $(J_n/J_n^2)^*$ is called normal module. An $S_n$-module $M$ is said to be reflexive if the natural map $$j:M\rightarrow \Hom_{S_n}(\Hom_{S_n}(M,S_n),S_n)$$ 
which sends $m\mapsto \varphi\in \Hom_{S_n}(M,S_n)$ to the map sending $\varphi\in \Hom_{S_n}(M,S_n)$ to $\varphi(m)\in S_n$ is an isomorphism. 

In the next lemma we give a minimal number of generators of $(\omega_{S_n})^*$. We also show that the conormal module of $J_n$ is a reflexive $S_n$-module.

\begin{lemma}\label{lemma1}
Consider the exact sequence (\ref{exact2}). Then 
\begin{equation}\label{exact3}
0\rightarrow \Hom_{S_n}(J_n / J_n^2,S_n)\xrightarrow{\pi_n^*} S_n^4\xrightarrow{\varphi_n^*} \Hom_{S_n}(\omega_{S_n},S_n)\rightarrow 0
\end{equation}
is exact. Moreover, $\Hom_{S_n}(\omega_{S_n},S_n)$ is the image of the transpose of $\varphi_n$ and $J_n/J_n^2$ is a reflexive $S_n$-module.
\end{lemma}
\begin{proof}
Using exact sequence (\ref{exact2}) and $\Hom_{S_n}(-,S_n)$ we get 
\begin{equation}
0\rightarrow \Hom_{S_n}(J_n / J_n^2,S_n)\rightarrow S_n^4\rightarrow \Hom_{S_n}(\omega_{S_n},S_n)\rightarrow {\rm Ext}^1_{S_n}(J_n/J_n^2,S_n)\rightarrow 0
\end{equation}

We claim that ${\rm Ext}_{S_n}^1(J_n/J_n^2,S_n)=0$. Resolutions (\ref{almostodd}) and (\ref{almosteven}) are of the form
\begin{equation}\label{exact5}
0\xrightarrow{}R\xrightarrow{d_3} R^{n+3}\xrightarrow{d_2}R^4\xrightarrow{d_1}R\rightarrow S_n\rightarrow 0.
\end{equation}

Set $(-)^!=\Hom_R(-,S_n)$. We obtain 

\begin{equation}
0\rightarrow \ker(d_2^!)\rightarrow S_n^{4} \xrightarrow{d_2^!}S^{n+3}_n\xrightarrow{d_3^!}S^n_n \rightarrow {\rm coker}(d_3^!)\rightarrow 0
\end{equation}
by applying $\Hom_R(-,S_n)$ on resolution (\ref{exact5}). We see that $\ker(d_2^!)={\rm Ext}_R^1(S_n,S_n)$. 
From the short exact sequence 

$$0\rightarrow J_n\rightarrow R\rightarrow R/J_n\rightarrow 0$$
one can get $\Hom_R(J_n,S_n)\simeq {\rm Ext}^1_R(S_n,S_n)$. Note that any $f\in \Hom_R(J_n,S_n)$ vanishes on $J_n^2$. Thus ${\rm Ext}_R^1(S_n,S_n)\simeq \Hom_{S_n}(J_n/J_n^2,S_n)$. We get an exact sequence

\begin{equation}\label{om}
0\rightarrow \Hom_{S_n}(J_n/J_n^2,S_n)\rightarrow S_n^{4} \xrightarrow{d_2^!}S_n^{n+3}\xrightarrow{d_3^!}S_n^n \rightarrow \omega_{S_n}\rightarrow 0.
\end{equation}
Then $(J_n/J_n^2)^*$ is a third syzygy module.  Consider a short exact sequence 

\begin{equation}
0\rightarrow (J_n/J_n^2)^* \xrightarrow{}S_n^4\xrightarrow{d_2^!}\im(d_2^!)\rightarrow 0.
\end{equation}
Then $\im(d_2^!)$ becomes a second syzygy module. By \cite[Section 2]{F} $Ass(\im(d_2^!))\subset Ass(S_n)$ and $\depth(\im(d_2^!)_p)\geq \min(2,\depth((S_n)_p))$ for all primes $p$ of $S_n$.  If $p\in Ass(S_n)$, then 
$(S_n)_p$ is complete intersection. For primes $p$ of height $1$, $(S_n)_p$ is a regular local ring since $S_n$ is a normal domain. Thus $(\omega_{S_n})_p\simeq (S_n)_p$ and $(J_n/J_n^2)_p$ is free $(S_n)_p$-module for ${\rm ht}(p)\leq 1$. Thus

\begin{equation}
0\xrightarrow{}(\omega_{S_n})_p\xrightarrow{(\varphi_n)_p} (S_n^4)_p\rightarrow (J_n/J_n^2)_p\rightarrow 0
\end{equation}
is split exact and  $({\rm Ext}_{S_n}^1(J_n/J_n^2,S_n))_p=0$ for  ${\rm ht}(p)\leq 1$. By Proposition \ref{prop1} $\im(d_2^!)\subset (\omega_{S_n})^*$ then we get the following commutative diagram

\[
\xymatrixrowsep{1.8pc} \xymatrixcolsep{2.2pc}
\xymatrix{
0\ar@{->}[r]^{}&(J_n/J_n^2)^*_p\ar@{->}[r]^{}\ar@{=}[d]^{}&(S_n^4)_p\ar@{=}[d]^{}\ar@{->}[r]^{(d_2^!)_p}&(\im(d_2^!))_p\ar@{->}[r]^{}\ar@{->}[d]^{}&0\\
0\ar@{->}[r]^{}&(J_n/J_n^2)^*_p\ar@{->}[r]^{}& (S_n^4)_p\ar@{->}[r]^{}&(\omega_{S_n}^*)_p\ar@{->}[r]^{}&0.\\
}
\]
By Snake's lemma $(\im(d_2^!))_p\simeq (\omega_{S_n})_p$ for ${\rm ht}(p)\leq 1$. For primes $p$ of height $\geq 2$ we have $\depth((\im(d_2^!)_p)\geq 2 $ since $\depth((S_n)_p)\geq 2$. Thus $\omega_{S_n}^*\simeq \im(d_2^!)$. By the change of basis we see that $d_2^!$ is the transpose of $\varphi_n$. Then $\im(d_2^!)\subset \omega_{S_n}^*$ which forces $\omega_{S_n}^*=\im(d_2^!)$, and this gives exact sequence (\ref{exact3}).

Applying $\Hom_{S_n}(-,S_n)$ on  exact sequence (\ref{exact3}), we obtain

\[0\rightarrow \Hom_{S_n}(\omega_{S_n}^*,S_n)\rightarrow S_n^4\rightarrow \Hom_{S_n}((J_n/J_n^2)^*,S_n)\rightarrow {\rm Ext}_{S_n}^1(\omega_{S_n}^*,S_n)\rightarrow 0.\]
But  ${\rm Ext}_{S_n}^1(\omega_{S_n}^*,S_n)=0$ as $\omega_{S_n}^*$ reflexive $S_n$-module. We obtain a commutative diagram

\[
\xymatrixrowsep{1.8pc} \xymatrixcolsep{2.2pc}
\xymatrix{
0\ar@{->}[r]^{}&\omega_{S_n}\ar@{->}[r]^{\varphi_n}\ar@{->}[d]^{\phi_{\omega_{S_n}}}&S_n^4\ar@{=}[d]^{}\ar@{->}[r]^{}&J_n/J_n^2\ar@{->}[r]^{}\ar@{->}[d]^{\phi_{J_n/J^2_n}}&0\\
0\ar@{->}[r]^{}&\omega_{S_n}^{**}\ar@{->}[r]^{\varphi^{**}_{n}}& S_n^4\ar@{->}[r]^{}&(J_n/J_n^2)^{**}\ar@{->}[r]^{}&0\\
}
\]
Here $\varphi_{\omega_{S_n}}$ is an isomorphism since $\omega_{S_n}$ is reflexive, and $\phi_{J_n/J_n^2}$ is injective as $J_n/J_n^2$ is torsion free module.  By Snake Lemma $\phi_{J_n/J_n^2}$ is an isomorphism. Thus $J_n/J_n^2$ is reflexive $S_n$-module.  
\end{proof}

\begin{corollary}
Normal module $(J_n/J_n^2)^*$ and $(\omega_{S_n})^*$ are maximal Cohen Macaulay modules.
\end{corollary}
\begin{proof}
We see that $(J_n/J_n^2)^*$ is a maximal Cohen Macaulay module since it is a third syzygy module of maximal Cohen Macaulay module $\omega_{S_n}$ by (\ref{om}).

Note that $\depth((\omega_{S_n})^*)\geq \min(\depth(J_n/J_n^2)+3,\depth(S_n))$ since $(\omega_{S_n})^*$ is a third syzygy of $J_n/J_n^2$.
 By the depth criteria on exact sequence (\ref{exact2}), $\depth(J_n/J_n^2)=\depth(\omega_{S_n})-1$. This forces $\depth(\omega_{S_n}^*)=\depth(S_n)$. Thus $(\omega_{S_n})^*$ is a maximal Cohen Macaulay module.
\end{proof}

Next we study resolution of conormal module $J_n/J_n^2$.
\begin{remark}
Consider resolutions (\ref{almostodd}) and (\ref{almosteven}).
Take maps $$\theta^{2m+1}_0=\begin{bmatrix}
w_1 & w_2 & \ldots & w_{2m+1}\\
v_{1,1} & v_{1,2} & \ldots & v_{1,2m+1}\\
v_{2,1} & v_{2,2} & \ldots & v_{2,2m+1} \\
v_{3,1} & v_{3,2} & \ldots & v_{3,2m+1} \\
\end{bmatrix},\;\;\;
\theta^{2m}_0=\begin{bmatrix}
w_1 & w_2 & \ldots & w_{2m} \\
v_{\{2,3\},1} & v_{\{2,3\},2} & \ldots & v_{\{2,3\},2m}\\
-v_{\{1,3\},1} & -v_{\{1,3\},2} & \ldots & -v_{\{1,3\},2m}\\
v_{\{1,2\},1} & v_{\{1,2\},2} & \ldots & v_{\{1,2\},2m}\\
\end{bmatrix}.
$$

By a simple homological algebra result there exists $\theta:\mathbb{F}^*\rightarrow \bigoplus\limits^{4}\mathbb{F}$ such that

\[
\xymatrixrowsep{1.3pc} \xymatrixcolsep{2.2pc}
\xymatrix{
R^4\ar@{->}[r]^{}&\ar@{->}[r]^{}S_n^4&0\\
R^n\ar@{->}[r]^{}\ar@{->}[u]^{\theta_0^n}&\omega_{S_n}\ar@{->}[r]^{}\ar@{->}[u]^{ \phi_n}&0.\\
}
\]

Then the mapping cone with respect to $\theta:\mathbb{F}^*\rightarrow \bigoplus\limits^{4} \mathbb{F}$ yields resolution of $J_n/J_n^2$ which need not be minimal. Since $\depth(J_n/J_n^2)=\depth(S_n)-1$, by Auslander-Buchsbaum formula, the projective dimension of $J_n/J_n^2$ is $4$.

\end{remark}

\section{Generic doublings of almost complete intersection}\label{GenD} 
In this section we discuss generic doublings of almost complete intersection. Let us recall generic doublings of $S_n$. In proof of Lemma \ref{lemma1} we see that $S_n$ is generically Gorenstein with canonical module $\omega_{S_n}$. Then by \cite[Prop. 3.3.18]{BH} $\omega_{S_n}$ is an ideal of $S_n$ and $S_n/\omega_{S_n}$ is a Gorenstein ring of codimension $4$.
 
 Let $\mathbb{F}$ be a minimal free resolution of $S_n$, and $h_1,h_2,h_3,h_4$ be generators of $\Hom_{S_n}(\omega_{S_n},S_n)$. Consider a bigger ring $\widetilde{R}=R[\alpha_0,\alpha_1,\alpha_2,\alpha_3]$. Set $\psi_0^n=\sum\limits_{i=0}^3\alpha_ih_i$. Denote $\widetilde{S}_n=\widetilde{R}/J_n\widetilde{R}$, and $\widetilde{\mathbb{F}}$ as a resolution of $\widetilde{S}_n$. Then $\widetilde{\mathbb{F}}^*$ is a minimal resolution of $\omega_{\widetilde{S}_n}$ since $\widetilde{S}_n$ is perfect. We lift map $\psi_0^n:\omega_{\widetilde{S}_n}\rightarrow \widetilde{S}_n$ to a map of complexes $\psi:\widetilde{\mathbb{F}}^*\rightarrow \widetilde{\mathbb{F}}$. Then the mapping cone with respect to map $\psi$ denoted as ${\rm Cone}(\psi)$ yields a resolution of Gorenstein ring $\widetilde{S}_n/\omega_{\widetilde{S}_n}$ of codimension $4$. In such case, we say that resolution of $\widetilde{S}_n/\omega_{\widetilde{S}_n}$ is constructed by generic doubling. 
 
\subsection{Generic doubling for n=2m+1} 
 Consider a bigger ring $\widetilde{R}=R[\alpha_0,\alpha_1,\alpha_2,\alpha_3]$. 
Set  $g_j=\alpha_0 w_j+\sum\limits_{i=1}^3 \alpha_i v_{i,j}$ for $1\leq j\leq n$ and 
$\psi^n_0=\begin{bmatrix}
g_1 & g_2 & \cdots & g_n\\
\end{bmatrix}
$. 
 Then by Proposition \ref{prop1}  map $\psi_0^n\in \Hom_{{\widetilde{S}}_n}(\omega_{\widetilde{S}_n},\widetilde{S}_n)$. Moreover $\psi_0^n$ is injective. Consider matrices $A=\begin{bmatrix}
 -\alpha_1\\
 \alpha_2\\
 -\alpha_3
 \end{bmatrix}
$ and  
 $M=\begin{bmatrix}
 U & A
 \end{bmatrix}
$. Define 

$$\psi_1^n:=\begin{bmatrix}
0& 0& \cdots &0 & \alpha_1 & \alpha_2 & \alpha_3\\
M_{1,2,;1,n} & M_{1,2;2,n}&\cdots & M_{1,2; 2m,n}& -\alpha_0 & 0 &0 \\
-M_{1,3;1,n} & -M_{1,3;2,n}&\cdots & -M_{1,3; 2m,n}& 0& -\alpha_0 & 0\\
M_{2,3;1,n} & M_{2,3;2,n}&\cdots & M_{2,3; 2m,n}& 0 & 0 &-\alpha_0\\
\end{bmatrix},\; \psi_2^n=-(\psi_1^n)^t,\;\psi_3^n=-(\psi_0^n)^t.  
$$
 
 Then we get the following commutaive diagram  
  
\[
\xymatrixrowsep{1.3pc} \xymatrixcolsep{2.2pc}
\xymatrix{
	\mathbb{F}:0\ar@{->}[r]^{}&\widetilde{F}\ar@{->}[r]^{d_3}&\widetilde{F}^*\oplus \widetilde{G}^*\ar@{->}[r]^{ d_2}& \widetilde{R}\oplus\widetilde{G}\ar@{->}[r]^{d_1}&\widetilde{R}\ar@{->}[r]^{}&\ar@{->}[r]^{}\widetilde{S}_n&0\\
	\mathbb{F}^*:0\ar@{->}[r]^{}&\widetilde{R}\ar@{->}[r]^{ d_1^*}\ar@{->}[u]^{ -(\psi_0^n)^t}&\widetilde{R}^*\oplus \widetilde{G}^*\ar@{->}[r]^{ d_2^*}\ar@{->}[u]^{ -(\psi_1^n)^t}& \widetilde{F}\oplus \widetilde{G}\ar@{->}[u]^{ \psi_1^n}\ar@{->}[r]^{ d_3^*}&\widetilde{F}\ar@{->}[r]^{}\ar@{->}[u]^{\psi_0^n}&\omega_{\widetilde{S}_n}\ar@{->}[r]^{}\ar@{->}[u]^{ \psi_0^n}&0.\\
}
\]
 Hence the mapping cone with respect to $\psi:\mathbb{F}^*\rightarrow \mathbb{F}$ is 
 
 \begin{equation}\label{con2m+1}
 {\rm Cone}(\psi): 0\xrightarrow{}\widetilde{R}\xrightarrow{\delta_4^n} {\begin{matrix}\widetilde{F}\\\bigoplus \\ \widetilde{R}^*\oplus \widetilde{G}^*\end{matrix}}\xrightarrow{\delta_3^n} {\begin{matrix}
 \widetilde{F}^*\oplus\widetilde{G}^*\\
 \bigoplus\\
 \widetilde{F}\oplus\widetilde{G}
 \end{matrix}}\xrightarrow{\delta_2^n}\begin{matrix}
 \widetilde{R}\oplus \widetilde{G}\\
 \bigoplus\\
 \widetilde{F}^*
 \end{matrix}\xrightarrow{\delta_1^n}\widetilde{R}\rightarrow \widetilde{R}/I_{n}\rightarrow 0
 \end{equation}
 where $I_{n}=J_{2m+1}\widetilde{R}+\im(\psi_0^n)$, $\delta_1^n=\begin{bmatrix}
 d_1 & \psi_0^n\\
 \end{bmatrix}$, $\delta_2^n=\begin{bmatrix}
 d_2 & \psi_1^n\\
 0 & -d_3^t
 \end{bmatrix}
$, $\delta_3^n=\begin{bmatrix}
d_3 & -{(\psi_1^n)}^t\\
0 & -d_2^t
\end{bmatrix} $
and $\delta_4^n=\begin{bmatrix}
-(\psi_0^n)^t\\
-d_1^t
\end{bmatrix}
$. Then ${\rm Cone}(\psi)$ is a Gorenstein ring of codimension $4$ with resolution of the form:

\begin{equation}\label{(1,9,16,9,1)}
{\rm Cone}(\psi):0\rightarrow \widetilde{R}\xrightarrow{\delta_4} \widetilde{R}^{n+4}\xrightarrow{\delta_3}\widetilde{R}^{2n+6}\xrightarrow{\delta_2} \widetilde{R}^{n+4}\xrightarrow{\delta_1}\widetilde{R}\rightarrow \widetilde{R}/I_{n}\rightarrow 0.
\end{equation}

\subsection{Generic doubling for n=2m}
Consider a bigger ring $\widetilde{R}=R[\alpha_0,\alpha_1,\alpha_2,\alpha_3]$. For $1\leq j\leq 2m$ set
\[g_j=\alpha_0 w_j+\alpha_1 v_{\{2,3\},j}-\alpha_2 v_{\{1,3\},j}+\alpha_3 v_{\{1,2\},j}.\]

Then by Proposition \ref{prop1}, $\psi_0^n=\begin{bmatrix}
g_1 & \cdots & g_{2m}\\
\end{bmatrix}
$ is a map  from $\omega_{\widetilde{S}_n}$ to  $\widetilde{S}_n$. Moreover $\psi_0^n$ is injective. Define $q_i:=u_{1i}\alpha_1+u_{2i}\alpha_2+u_{3i}\alpha_3$. Choose matrices  $$A_n=\begin{bmatrix}
0& 0 & 0\\
\alpha_2 & 0 &-\alpha_3\\
-\alpha_1& \alpha_3& 0\\
0 & -\alpha_2 & \alpha_1
\end{bmatrix},\;\;
B^n=\begin{bmatrix} B_1&\cdots & B_n \end{bmatrix}, \;\text{where}\;\; 
B_i=\begin{bmatrix}
q_i\\
u_{1i}\alpha_0\\
-u_{2i}\alpha_0\\
u_{3i}\alpha_0\\
\end{bmatrix}
$$
such that $\psi_1^n=\begin{bmatrix}
B^n & A^n\\
\end{bmatrix}
$. Set $\psi_2=-(\psi_1)^t$ and  $\psi_3=-(\psi_0)^t$.

\[
\xymatrixrowsep{1.3pc} \xymatrixcolsep{2.2pc}
\xymatrix{
	\mathbb{F}:0\ar@{->}[r]^{}&\widetilde{F}\ar@{->}[r]^{d_3}&\widetilde{F}^*\oplus \widetilde{G}^*\ar@{->}[r]^{ d_2}& \widetilde{G}^*\oplus\widetilde{R}\ar@{->}[r]^{d_1}&\widetilde{R}\ar@{->}[r]^{}&\ar@{->}[r]^{}\widetilde{S}_n&0\\
	\mathbb{F}^*:0\ar@{->}[r]^{}&\widetilde{R}\ar@{->}[r]^{ d_1^*}\ar@{->}[u]^{ -(\psi_0^n)^t}&\widetilde{G}\oplus \widetilde{R}\ar@{->}[r]^{ d_2^*}\ar@{->}[u]^{ -(\psi_1^n)^t}& \widetilde{F}\oplus \widetilde{G}\ar@{->}[u]^{ \psi_1^n}\ar@{->}[r]^{ d_3^*}&\widetilde{F}^*\ar@{->}[r]^{}\ar@{->}[u]^{\psi_0^n}&\omega_{\widetilde{S}_n}\ar@{->}[r]^{}\ar@{->}[u]^{ \psi_0^n}&0.\\
}
\]

Then the mapping cone with respect to $\psi:\mathbb{F}^*\rightarrow \mathbb{F}$ is 
 
 \begin{equation}\label{con2m}
 {\rm Cone}(\psi): 0\xrightarrow{}\widetilde{R}\xrightarrow{\delta_4^n} {\begin{matrix}\widetilde{F}\\\bigoplus \\ \widetilde{R}\oplus \widetilde{G}\end{matrix}}\xrightarrow{\delta_3^n} {\begin{matrix}
 \widetilde{F}^*\oplus\widetilde{G}^*\\
 \bigoplus\\
 \widetilde{F}\oplus\widetilde{G}
 \end{matrix}}\xrightarrow{\delta_2^n}\begin{matrix}
 \widetilde{G}^*\oplus \widetilde{R}\\
 \bigoplus\\
 \widetilde{F}^*
 \end{matrix}\xrightarrow{\delta_1^n}\widetilde{F}\rightarrow \widetilde{R}/I_{n}\rightarrow 0
 \end{equation}
 where $I_{n}=J_{2m+1}\widetilde{R}+\im(\psi_0^n)$, $\delta_1^n=\begin{bmatrix}
 d_1 & \psi_0^n\\
 \end{bmatrix}$, $\delta_2^n=\begin{bmatrix}
 d_2 & \psi_1^n\\
 0 & -d_3^t
 \end{bmatrix}
$, $\delta_3^n=\begin{bmatrix}
d_3 & -{(\psi_1^n)}^t\\
0 & -d_2^t
\end{bmatrix} $
and $\delta_4^n=\begin{bmatrix}
-(\psi_0^n)^t\\
-d_1^t
\end{bmatrix}
$. Therefore ${\rm Cone}(\psi)$ is a Gorenstein ring of codimension $4$ with resolution of the form:

\begin{equation}\label{(1,8,14,8,1)}
{\rm Cone}(\psi):0\rightarrow \widetilde{R}\xrightarrow{\delta_4^n} \widetilde{R}^{n+4}\xrightarrow{\delta_3^n}\widetilde{R}^{2n+6}\xrightarrow{\delta_2^n} \widetilde{R}^{n+4}\xrightarrow{\delta_1^n}\widetilde{R}\rightarrow \widetilde{R}/I_{2m}\rightarrow 0. 
\end{equation}

\section{Spinor coordinates of (1,9,16,9,1)} 
We study spinor coordinates of Gorenstein ring with $9$ generators obtained from generic doubling of almost complete intersection ring $S_5$. 
We consider resolution given in (\ref{(1,9,16,9,1)}) for $n=5$.
In \cite[Theorem 2]{CLW}, there is 
a hyperbolic basis of $\widetilde{R}^{16}$ 
say $\{e_1,\ldots,e_8,e_{-1},\ldots,e_{-8}\}$ with hyperbolic pairs $\{e_i,e_{-i}\}$. Denote columns of $\delta_2^5$ with respect to $e_i$ and $e_{-i}$ as $i$ and $\bar{i}$ respectively. In \cite[Theorem 2]{CLW} spinor coordinates are denoted as $(\widetilde{a}_3)_K$  where $K\subset \{\pm 1,\ldots,\pm 8\}$ of cardinality $8$ with odd number of $\bar{i}$. In Proposition \ref{spin9} we see that there are $5$ spinor coordinates among minimal generators of ideal $I_5$.

\begin{proposition}\label{spin9}
There are $5$ spinor coordinates of resolution (\ref{(1,9,16,9,1)}) for $n=5$ which are among the minimal generators of ideal $I_5$.
\end{proposition}  
\begin{proof}
We use Macaulay 2 \cite{M2} to compute $8\times 8$ minors of matrix $\delta_2^5$. Denote $\mathcal{M}^{K}_{L}$ as $8\times 8$ minors of matrix $\delta_2^5$ involving $K$ rows and $L$ columns. Then $\mathcal{M}^{2,3,4,5,6,7,8,9}_{\bar{1},2,3,4,5,6,7,8}=\pm g_1^2x_1$,

$$\mathcal{M}^{2,3,4,5,6,7,8,9}_{1,\bar{2},3,4,5,6,7,8}=\pm g_2^2x_1, \;\; \mathcal{M}^{2,3,4,5,6,7,8,9}_{1,2,\bar{3},4,5,6,7,8}=\pm g_3^2x_1, $$ 

$$\mathcal{M}^{2,3,4,5,6,7,8,9}_{1,2,3,\bar{4},5,6,7,8}=\pm g_4^2x_1, \;\; \mathcal{M}^{2,3,4,5,6,7,8,9}_{1,2,3,4,\bar{5},6,7,8}=\pm g_5^2x_1. $$

Then by \cite[Theorem 2]{CLW} spinor coordinates are  $$(\tilde{a}_3)_{\bar{1},2,3,4,5,6,7,8}=\pm g_1,\;\;
(\widetilde{a}_3)_{1,\bar{2},3,4,5,6,7,8}=\pm g_2, \;\; (\widetilde{a}_3)_{1,2,\bar{3},4,5,6,7,8}=\pm g_3, $$ 

$$(\tilde{a}_3)_{1,2,3,\bar{4},5,6,7,8}=\pm g_4, \;\; (\tilde{a}_3)_{1,2,3,4,\bar{5},6,7,8}=\pm g_5. $$
These $5$ spinor coordinates are part of minimal generating set of $I_5$ but not contained in $J_5$.
 
Other non zero spinor coordinates are of the form $\pm \alpha_ix_j\pm \alpha_kx_l$ and ${\rm Pf}(\hat{i})x_j$ where ${\rm Pf}(\hat{i})$ is the Pfaffian of $C$ obtained by omitting ith row and column. These other spinor coordinates are not among the minimal generating set of $I_5$.  
\end{proof}

\begin{remark}
Consider the following examples:
\begin{enumerate}[{\rm 1)}]
\item Let $R$ be a polynomial ring in $9$ variables on the entries of $3\times 3$ generic  matrix $X$. The ideal $I$ is generated by $2\times 2$ minors of matrix $X$. \\
\item Let $R$ be a polynomial ring in $8$ variables with ideal $I$ generated by the equation of Segre embedding of $\mathbb{P}^1\times \mathbb{P}^1\times \mathbb{P}^1$ into $\mathbb{P}^7$. 
\end{enumerate}
In both examples, $I$ is Gorenstein ideal of codimension $4$ with $9$ generators where none of the spinor coordinates are among the minimal generators of $I$ by \cite[Example 3]{CLW}.

Suppose 1) and 2) are specializations of resolution (\ref{(1,9,16,9,1)}) of $I_5$. Since specialization is a ring homomorphism which maps minors of the matrix to the minors of matrix, and thus maps spinor coordinates of one resolution to spinor coordinates of the other. In (1) and (2) there are none of the spinor coordinates are among minimal generators of $I$. This forces none of the spinors coordinates of resolution of $I_5$ in (\ref{(1,9,16,9,1)}) to be among minimal generators of $I_5$ which is contradiction by Proposition \ref{spin9}.
\end{remark} 
 
\section{Spinor coordinates of generic doubling of (1,4,7,4)}\label{spin8}
 In this section we discuss spinor coordinates of resolution (\ref{(1,8,14,8,1)}) of Gorenstein ring with $8$ generators obtained from generic doubling of $(1,4,7,4)$. We investigates number of spinor coordinates which are among the minimal generators of ideal $I_4$. 
 
 By \cite[Theorem 2]{CLW} there exists a hyperbolic basis of $\widetilde{R}^{14}$ in resolution (\ref{(1,8,14,8,1)}) for $m=2$, say $\{e_1,\ldots,e_7,e_{-1},\ldots,e_{-7}\}$ with hyperbolic pairs $\{e_i,e_{-i}\}$. Denote columns of $\delta_2^4$ with respect to $e_i$ and $e_{-i}$ as $i$ and $\bar{i}$ respectively. In \cite[Theorem 2]{CLW} spinor coordinates are denoted as $(\widetilde{a}_3)_K$  where $K\subset \{\pm 1,\ldots,\pm 7\}$ with cardinality of $K$ is $7$ with even number of $\bar{i}$. We find the number of spinor coordinates are among the minimal generators of ideal $I_4$.

\begin{proposition}\label{spinC8}
For $m=2$ in resolution (\ref{(1,8,14,8,1)}), at least $4$ spinor coordinates among the minimal generators of $I_4$. 
\end{proposition} 
 \begin{proof}
We use Macaulay 2 \cite{M2} to calculate $7\times 7$ minors of $\delta_2^4$ mentioned in resolution (\ref{(1,8,14,8,1)}) for $m=2$. Denote $\mathcal{M}^{K}_{L}$ as $7\times 7$ minors of matrix $\delta_2^4$ involving $K$ rows  and $L$ columns. Then 

$$\mathcal{M}^{2,3,4,5,6,7,8}_{\bar{1},\bar{2},\bar{3},\bar{4},\bar{5},\bar{6},7}=\pm g_1^2x_1, \;\; \mathcal{M}^{2,3,4,5,6,7,8}_{\bar{1},\bar{2},\bar{3},\bar{4},\bar{5},6,\bar{7}}=\pm g_2^2x_1, $$ 

$$\mathcal{M}^{2,3,4,5,6,7,8}_{\bar{1},\bar{2},\bar{3},\bar{4},5,\bar{6},\bar{7}}=\pm g_3^2x_1, \;\; \mathcal{M}^{2,3,4,5,6,7,8}_{\bar{1},\bar{2},\bar{3},4,\bar{5},\bar{6},\bar{7}}=\pm g_4^2x_1. $$

Then by \cite[Theorem 2]{CLW} spinor coordinates with respect to above $K$ columns are:

$$(\tilde{a}_3)_{\bar{1},\bar{2},\bar{3},\bar{4},\bar{5},\bar{6},7}=\pm g_1,\;\; (\tilde{a}_3)_{\bar{1},\bar{2},\bar{3},\bar{4},\bar{5},6,\bar{7}}=\pm g_2,$$

$$(\tilde{a}_3)_{\bar{1},\bar{2},\bar{3},\bar{4},5,\bar{6},\bar{7}}=\pm g_3,\;\; (\tilde{a}_3)_{\bar{1},\bar{2},\bar{3},4,\bar{5},\bar{6},\bar{7}}=\pm g_4,$$

These $4$ spinor coordinates are among minimal generators of ideal $I_4$ but not contained in $J_4$. 
\end{proof}
 
\section*{Acknowledgement} 
The author thanks Jerzy Weyman for helpful discussions. The author acknowledges support of Fulbright-Nehru fellowship.

 \bibliographystyle{amsplain}

\providecommand{\MR}[1]{\mbox{\href{http://www.ams.org/mathscinet-getitem?mr=#1}{#1}}}
\renewcommand{\MR}[1]{\mbox{\href{http://www.ams.org/mathscinet-getitem?mr=#1}{#1}}}
\providecommand{\href}[2]{#2}

\end{document}